\documentclass[leqno,draft,12pt]{amsart}
\usepackage{amssymb}
\usepackage{amsmath}
\usepackage{enumerate}
\usepackage{amsfonts}
\usepackage{color}
\newtheorem{theorem}{Theorem}[section]
\newtheorem{lemma}[theorem]{Lemma}
\newtheorem{proposition}[theorem]{Proposition}
\newtheorem{corollary}[theorem]{Corollary}

\theoremstyle{remark}

\theoremstyle{remark}
\newtheorem*{note}{Remark}

\theoremstyle{definition}
\newtheorem*{example}{Example}

\numberwithin{equation}{section}

 \DeclareMathOperator{\Ad}{Ad}
\DeclareMathOperator{\ad}{ad} 
\DeclareMathOperator{\Var}{Var}

\DeclareMathOperator{\diag}{diag}
\newcommand{\e}{\text{\bf E}}

\newcommand{\N}{\mathbb{N}}

\newcommand{\R}{\mathbb{R}}

\newcommand{\C}{\mathbb{C}}

\newcommand{\one}{\mathbf 1}
\newcommand{\tr}{\mathrm{t}}

\renewcommand\o{\overline}

\def\mc#1{\mathcal#1 }
\def\tl{\tilde}
\def\dbar{\partial}
\def\f{\phi}
\def\Lam{\Lambda}
\def\lam{\lambda}
\def\sg{\sigma}
\def\ve{\varepsilon}

\def\ali{\aligned}
\def\eal{\endaligned}
\def\8{\infty}
\def\blab#1{\begin{equation}\label{#1}}
\def\elab{\end{equation}}
\def\refp#1{~(\ref{#1}) on page~\pageref{#1}}

\def\refp#1{~(\ref{#1})}
\def\refn#1{~\ref{#1}}
\def\lp{\left(}
\def\rp{\right)}

\def\fn#1#2{\frac{#1}{#2}}

\def\g{\gamma }

\def\blab#1{\begin{equation}\label{#1}}
\def\elab{\end{equation}}

\def\fn#1#2{\frac{#1}{#2}}
\def\mf#1{\mathfrak{#1}}
\def\rest#1{\big |_{#1}}
\def\fac{\,\forall\, }
\def\bma{\begin{bmatrix}}
\def\ebm{\end{bmatrix}}

\title[Evolution and Poisson kernels]{The evolution and Poisson kernels on nilpotent meta-Abelian groups}

\author[R. Penney]{Richard Penney}
\address{Department of Mathematics\\
Purdue University\\
150 N. U\-ni\-ver\-si\-ty St\\
West Lafayette, IN 47907, USA} \email{rcp@math.purdue.edu}

\author[R. Urban]{Roman Urban}
\address{Institute of Mathematics\\
Wroclaw University\\
Plac Grunwaldzki 2/4\\
50-384 Wroclaw, Poland} \email{urban@math.uni.wroc.pl}

\begin{document}
\begin{abstract}
Let $S$ be a semi direct product $S=N\rtimes A$ where $N$ is a connected and simply connected, non-abelian,
nilpotent meta-abelian Lie group and $A$ is isomorphic with $\R^k,$ $k>1.$ We consider a class of second order left-invariant differential
operators on $S$ of the form $\mathcal L_\alpha=L^a+\Delta_\alpha,$ where $\alpha\in\R^k,$ and for each $a\in\R^k,$ $L^a$ is left-invariant second order differential operator on $N$ and $\Delta_\alpha=\Delta-\langle\alpha,\nabla\rangle,$ where $\Delta$ is the usual Laplacian on $\R^k.$ Using some probabilistic techniques (e.g., skew-product formulas for diffusions on $S$ and $N$ respectively) we obtain an upper estimate for the transition probabilities of the evolution on $N$ generated by $L^{\sigma(t)},$ where $\sigma$ is a continuous function from $[0,\infty)$ to $\R^k.$ We also give an upper bound for the Poisson kernel for $\mathcal L_\alpha.$
\end{abstract}
\maketitle
\section{Introduction}
\subsection{The evolution kernel on $NA$ groups}\label{introduction}

We say that a solvable Lie group  $S$ is an $NA$ group if it is a semi-direct product $S=N\rtimes A$ where $N$ is a connected and simply connected
nilpotent Lie group and $A$ is isomorphic with $\R^k.$ There is a remarkable probabilistic formula (formula\refp{probform} below) for the heat semi-group defined by a fairly general second order elliptic, or even degenerate elliptic, left invariant, differential operator  on an $NA$ group that has long played a central role in their analysis. (See  \cite{D,DHstudia,DHZ,DHU,PUcm,PUpota,JEE} for example.)

To describe this formula in our context, let $\mf a$ and $\mf n$ be the Lie algebras for $A$ and $N$ respectively. In general, we identify connected, simply connected nilpotent Lie groups with their Lie algebra using the exponential map so that in particular, $A$ and $N$ are identified with $\mf a$ and $\mf n$.
We assume that there is a  basis $\mc B=\{X_1,\ldots,X_d\}$ for $\mathfrak{n}$ that diagonalizes the $A$-action.  We typically think of the $X_i$ as left-invariant differential operators on $N$.  When thought of as left-invariant operators on $S$ they are denoted $\tl X_i$.  Thus, for $f\in C^\8(S)$
\blab{Xi}
\tl X_i f(n,a)=e^{\lambda_i(a)}X_if(n,a)
\elab
where $\lambda_i\in\mathfrak{a}^*$ is the root functional corresponding to $X_i$, i.e., $[H,X_i]=\lambda_i(H)X_i,$ for all $H\in\mf a$.  We also choose a basis $\{A_1,\dots,A_k\}$ for $\mf a$, which we use to identify $\mf a$ with $\R^k$.

The Euclidean space $\R^k$ is endowed with the usual scalar product $\langle\cdot,\cdot\rangle$ and the corresponding $\ell^2$ norm $\|\cdot\|.$ For the vector $x\in\R^k$ we write $x^2=x\cdot x=\langle x,x\rangle=\sum_{i=1}^k x_i^2.$ By $\|\cdot\|_\infty,$ we denote the $\ell^\infty$ norm $\|x\|_\infty=\max_{1\leq i\leq k}|x_i|.$

For $\alpha=(\alpha_1,\ldots,\alpha_k)\in\R^k,$ let
\begin{equation}\label{defofl}
\ali
\mathcal{L}_{\alpha} &=\sum_{j=1}^d\tl X_j^2+\sum_{j=1}^k (A_j^2-2\alpha_jA_j)\\
&=\sum_{j=1}^de^{2\lambda_j(a)}X_j^2+\Delta_\alpha,\\
\eal
\end{equation}
where
\begin{equation}\label{defofdelta}
\Delta_\alpha=\sum_{j=1}^k(\partial_{a_j}^2-2\alpha_j\partial_{a_j}).
\end{equation}

For $a\in\R^k$ we let
$$
\mathcal L_N^a=\sum_{j=1}^de^{2\lambda_j(a)}X_j^2.
$$
For $\sg\in C^\8([0,\8),\R^k)$ and $s<t<\8$, let
$
P^\sg_{t,s}(x),\,x\in N
$
be the fundamental solution for the operator
\begin{equation}\label{opL}
L=\dbar_s+\mathcal L_N^{\sg(s)}.
\end{equation}
Thus $P^\sg_{t,s}$ is a non-negative function on $N$ such that
\blab{IntPsg}
\int_N P^\sg_{t,s}(x)dx=1
\elab
and, for $s<u<t,$
\blab{SgP}
P^\sg_{t,u}*P^\sg_{u,s}=P^\sg_{t,s}.
\elab
 Moreover, if $\f\in C^\8_c(N)$ then
 \blab{defOfUsg}
\f*P^\sg_{t,s}\equiv U^\sg_{s,t}(\f)
\elab
 is the solution to the Dirichlet problem on $N\times (s,\8)$ with  boundary data $\f,$ i.e.,
\blab{DefOfUst}
LU^\sg_{s,t}(\f)=0 \text{ on } N\times (0,t),\, \lim_{t\to s^+}U^\sg_{s,t}(\f)(x)=\f(x).
\elab
(For the existence of $P^\sg_{t,s}$ see~\cite{DHstudia,Tanabe}.)  In probabilistic terms, $P^\sg_{t,s}$ is the  kernel for the evolution defined by the time dependent family of operators $\mathcal L_N^{\sg_s}$.

Of course, $U^\sg_{s,t}$ can also be thought of as an integral operator. With obvious abuse of notation,  we  denote the corresponding kernel by $P^\sg_{t,s}(x;y)$. Thus
$$
P^\sg_{t,s}(x;y)=P^\sg_{t,s}(y^{-1}x).
$$
For $f\in C_c(N\times\R^k)$ and $t\geq 0,$ we put
\begin{equation}\label{probform}
T_tf(x,a)=\e_aU^\sigma_{0,t}f(x,\sigma_t)=\e_a(f*_{N}P^\sigma_{t,0})(x,\sigma(t)),
\end{equation}
where the expectation is taken with respect to the distribution of
the process $\sigma(t)$ (Brownian motion with drift) in $\R^k$ with the generator
$\Delta_\alpha.$ The operator $U^\sigma(0,t)$ acts on the
first variable of the function $f$ (as a convolution operator).

We have the following
\begin{theorem}\label{dmuh}
The family $T_t$ defined in \eqref{probform} is the semigroup of
operators generated by $\mathcal L_\alpha.$ That is
\begin{equation*}
\partial_tT_tf=\mathcal L_\alpha T_tf
\end{equation*}
and
\begin{equation*}
\lim_{t\to 0}T_tf=f.
\end{equation*}
\end{theorem}

Of course, the Brownian motion with a drift is an extremely well understood object. Clearly, then, a good understanding of $P^\sg_{t,s}$ is key to understanding the heat semi-group as well as objects derived from it, such as the Poisson kernel.

It is not difficult to give an explicit formula for $P^\sg_{t,s}$ in the case that $N$ is abelian. (See Proposition\refn{RnPsig} below.)   Our first main result is a skew-product formula for $P^\sg_{t,s}$ (Theorem~\ref{DecPs2}) similar to formula\refp{probform} that describes $P^\sg_{t,s}$ on a meta-abelian group.
Specifically, we assume that
\begin{equation*}
N=M\rtimes V
\end{equation*}
where $M$ and $V$ are abelian Lie groups with the corresponding Lie algebras $\mathfrak m$ and $\mathfrak v$.  Let $\mc B_1=\{Y_1,\ldots,Y_m\}$ and $\mc B_2=\{X_1,\ldots,X_n\}$ be ordered bases for $\mathfrak m$ and $\mathfrak v$ respectively such that $\mc B'=\mc B_1\cup \mc B_2$ forms an ordered Jordan-H\"older basis for the Lie algebra $\mathfrak n$ of $N$, ordered so that the matrix of $\ad_X$ in this basis is strictly lower triangular for all $X\in \mathfrak n$. We  use $\mc B'$ in place of the basis $\mc B$ mentioned above\refp{Xi}.  Hence, in this case,
\begin{equation}\label{defofl2}
\begin{split}
\mathcal{L}_\alpha=&\Delta_\alpha+\sum_{j=1}^m e^{2\xi_j(a)}Y_j^2+\sum_{j=1}^ne^{2\vartheta_j(a)}X_j^2\\
=&\Delta_\alpha+\mathcal L_N^a.
\end{split}
\end{equation}
where $\xi_1,\ldots,\xi_m$ and $\vartheta_1,\ldots,\vartheta_n$ are the root functionals in $\mf a^*$ corresponding to the bases $\mc B_1$ and $\mc B_2$ respectively.

The time dependent family of operators
\begin{equation}\label{operators}
\begin{split}
\mathcal L_V^{\sg,t}&=\sum_{j=1}^{n}e^{2\vartheta_j(\sg(t))}X_j^2,
\end{split}
\end{equation}
gives rise to an evolution on $V=\R^n$ that  is described by a kernel $P^{V,\sigma}_{t,s}$ which may be explicitly computed, since $V$ is abelian.  In fact, it turns out that the process $\eta(t)$ generated by $\mathcal L_V^{\sg,t}$ has coordinates $\eta_j(t)$ which are independent Brownian motions with time shifted by
\blab{TShift}
A^\sg_{V,i}(s,t)=\int_s^te^{2\vartheta_{j}(\sigma(u))}du.
\elab

For $\eta\in C^\8([0,\8), V)$ let
$$
\mathcal L_M^{\sg,\eta,t}=\sum_{j=1}^{m}e^{2\xi_j(\sg(t))}(\Ad(\eta(t))Y_j)^2.
$$
This family of operators
gives rise to an evolution on $M=\R^m$ that  is described by a kernel $P^{M,\sigma,\eta}_{t,s}$ which may also be explicitly computed. Specifically, for $a\in A$, let $S(a)$ be the $m\times m$ matrix
$$
S(a)=\diag\left[e^{\xi_1(a)},\ldots,e^{\xi_m(a)}\right].
$$

For $v\in V$,  we identify $\Ad(v)\rest{\mf m}$ with the $m\times m$ matrix of this linear transformation with respect to the basis $\mc B_1$.  Let
\begin{equation*}
[a^{\sigma,\eta}_M(t)]=2\left[\Ad(\eta(t))\rest{\mf m}S^\sigma(t)\right]\left[\Ad(\eta(t))\rest{\mf m}S^\sigma(t)\right]^*,
\end{equation*}
where
\begin{equation*}
S^\sigma(t):=S(\sigma(t)),
\end{equation*}
and
\begin{equation*}
A^{\sigma,\eta}_M(s,t)=\int_s^t a_M^{\sigma,\eta}(u) du.
\end{equation*}
Finally, for a $d\times d$ invertible matrix $A$ we set
\begin{equation}\label{bede}
\mathcal B(A)(x)=\frac{1}{2} A^{-1}x\cdot x\text{ and }
\mathcal D(A)=(2\pi)^{-\frac{d}{2}}(\det A)^{-\frac{1}{2}}.
\end{equation}
We prove in \S\ref{eonM} that
for $m^1,m^2\in M=\R^{m},$
\begin{equation}\label{PMform}
P^{M,\sigma,\eta}_{t,s}(m^1,m^2)=\mathcal D(A^{\sigma,\eta}_M(t,s))e^{-\mathcal B(A^{\sigma,\eta}_M(t,s))(m^1-m^2)}.
\end{equation}

Our main tool is the following theorem.  To the best of our knowledge, this result represents the first known formula for the evolution defined by $P^\sg_{t,s}$ for a non-abelian $N$ other than the similar result for the Heisenberg group from our work~\cite{JEE}.

\begin{theorem}\label{DecPs2}
Let $N=M\rtimes V.$ For  every $m\in M$ and $v\in V$ and a.e.(with respect to the corresponding Wiener measure) trajectory $\sigma$ of the process generated by $\Delta_\alpha,$
\begin{multline*}
\int_NP^{\sigma}_{t,0}(m,v;m^\prime,v^\prime)f(m^\prime,v^\prime)dm^\prime dv^\prime\\=\int_M P^{M,\sigma,\eta}_{t,0}(m,m^\prime)f(m,\eta(t))dm^\prime d\mathbf W^{V,\sigma}_{y}(\eta)\\
=\int_M \mathcal D(A^{\sigma,\eta}_M(0,t))e^{-\mathcal B(A^{\sigma,\eta}_M(0,t))(m-m^\prime)}f(m^\prime,\eta(t))dm^\prime d\mathbf W^{V,\sigma}_{y}(\eta),
\end{multline*}
where $\mathbf W^{V,\sigma}_{y}$ is the product of $n$ one-dimensional Wiener measures transformed according to \eqref{TShift}, i.e., for the trajectory $\eta(t)=(\eta_1(t),\ldots,\eta_n(t))$ its coordinates $\eta_i(t)$ are the one-dimensional Brownian motions $b_i(t)$ starting from $v_i$ with their time changed by $A_{V,i}^\sigma(0,t),$ i.e.,
\begin{equation*}
\eta_i(t)=b_i(A_{V,i}^\sigma(0,t)).
\end{equation*}
\end{theorem}

Theorem\refn{DecPs2}  yields a new estimate on $P^{\sigma}_{t,0}$ which is our second main result. In order to state this result let, for a continuous function $\sigma:[0,\infty)\to A=\R^k,$
\begin{equation}\label{ajsigma}
\begin{split}
A_{M,i}^\sigma(s,t)=&\int_s^te^{2\xi_{i}(\sigma(u))}du,\;\;i=1,\ldots,m,\\
A_{V,j}^\sigma(s,t)=&\int_s^te^{2\vartheta_{j}(\sigma(u))}du,\;\;j=1,\ldots,n,
\end{split}
\end{equation}
and
\begin{align*}
A_{M,\Sigma}^{\sigma}(s,t)=&\sum_{i=1}^mA_{M,j}^\sigma(s,t),&
A_{V,\Sigma}^{\sigma}(s,t)=&\sum_{j=1}^nA_{V,j}^\sigma(s,t),\\
A_{M,\Pi}^\sigma(s,t)=&\prod_{i=1}^mA_{M,j}^\sigma(s,t),&
A_{V,\Pi}^\sigma(s,t)=&\prod_{j=1}^nA_{V,j}^\sigma(s,t).
\end{align*}
We also set
\begin{equation*}
\begin{split}
A_{N,\Pi}^\sigma(0,t)&=A_{M,\Pi}^\sigma(0,t)A_{V,\Pi}^{\sigma}(0,t),\\
A_{N,\Sigma}^\sigma(0,t)&=A_{M,\Sigma}^{\sigma}(0,t)+A_{V,\Sigma}^\sigma(0,t).
\end{split}
\end{equation*}
We also let $k_o$ be the smallest non-negative integer such that
\blab{DefOfko}
(\ad_X)^{k_o+1}\rest{\mf{m}}=0,\fac X\in\mf{v}.
\elab
Note that if $k_o=0$, then $\mf{v}$ centralizes $\mf{m}$; hence $N$ is abelian. Thus our hypotheses imply that $k_o>0$.

The following theorem is a simplified version of Theorem~\ref{PreEst} which is one of our main results.

For, $a,b\in\R$ we write $a\wedge b=\min\{a,b\}.$
\begin{theorem}\label{ubpsigma} There are positive constants $C,D$  such that for all $(m,v)\in N=M\rtimes V,$
\begin{multline*}
P^\sigma_{t,0}(m,v)\leq C(A_{N,\Pi}^\sigma(0,t))^{-1\slash 2}(\|m\|^{1\slash (2k_0)}+1+A_{V,\Sigma}(0,t)^{1\slash 2})\\\times\exp\left(-D\frac{\|v\|^2}{A_{V,\Sigma}^\sigma(0,t)}-D\frac{\|m\|^{1\slash k_0}\wedge\|m\|^2}{A_{N,\Sigma}^\sigma(0,t)}\right).
\end{multline*}
\end{theorem}

It is interesting to compare this result with what is known in the general case.  The best general result that we are aware of in the literature is, when specialized to our current context, Theorem~\ref{evolutioninrn} below. (See~\cite{DHstudia,DHU} and~\cite{PUcm}.)  Theorem~\ref{ubpsigma} is an improvement in two respects.  First, it applies to all $(m,v)$, not just points in  a compact set not containing $e$. Secondly, it does not contain a term such as $\frac{\tau(x)}{4}$ in the exponent which is large when $\tau(x)$ is large. Of course, this term will be eventually dominated by the $\tau(x)^2$, but the point at which this domination takes place depends on the sizes of both $\tau(x)$ and of $A_{N,\Sigma}^\sigma(0,t)$ which are very hard to control.  We conjecture that a result such as Theorem~\ref{ubpsigma} holds in general.

\begin{theorem}\label{evolutioninrn}
Let $K\subset N$ be closed and $e\not\in K.$ Then there exist positive constants $C_1$, $C_2$, and $c$ such
that for every $x\in K$ and for every $t,$
$$
\ali
&P^\sigma_{t,0}(x)\leq C_1\left(\int_0^t(A_{N,\Pi}^\sigma(0,t))^{2\slash c}du\right)^{-c\slash 2}\exp\left(\frac{\tau(x)}{4}-\frac{\tau(x)^2}{C_2A_{N,\Sigma}^\sigma(0,t)}\right),
\eal
$$
where $\tau$ is a subadditive norm which is smooth on $N\setminus\{e\}$.
\end{theorem}

\subsection{Poisson kernel for $\mathcal L_\alpha$} As mentioned above, we expect improved estimates for $P^\sg_{t,0}$ to yield better estimates for objects derived from the heat semi-group such as the Poisson kernel.
As an illustration of this we use Theorem~\ref{ubpsigma} to prove Theorem~\ref{newupper} below that, in the current context, improves the estimates from~\cite{PUcm} and~\cite{PUpota}. (See \S~\ref{example}.) This result is our final ``main result.'' To state it we again require some notation.

Define
\begin{equation}\label{DefOfrho0}
\rho_0=\sum_{j=1}^d\lambda_j
\end{equation}
and set
\begin{equation}\label{DefOfX}
\chi(g)=\det(\Ad(g))=e^{\rho_0(a)},
\end{equation}
where
\begin{equation*}
\Ad(g)s=gsg^{-1},\;\;s\in S.
\end{equation*}
Let $ds$ be left-invariant Haar measure on $S.$ We have
\begin{equation*}
\int_Sf(sg)ds=\chi(g)^{-1}\int_Sf(s)ds.
\end{equation*}
Let
\begin{equation*}
A^+=\mathrm{Int}\{a\in\R^k:\lambda_j(a)\geq 0\text{ for }1\leq j\leq r\}.
\end{equation*}

If $\alpha\in A^+$ then there exists a \textit{Poisson kernel} $\nu$ for $\mathcal L_\alpha,$ \cite{D}. That is, there is a $C^\infty$ function $\nu$ on $N$ such that every bounded $\mathcal L_\alpha$-harmonic function $F$ on $S$ may be written as a {\em Poisson integral} against a bounded function $f$ on $S\slash A=N,$
\begin{equation*}
F(g)=\int_{S\slash A}f(gy)\nu(y)dy=\int_Nf(y)\check\nu^a(y^{-1}x)dy,\text{ where $g=(x,a),$}
\end{equation*}
and
\begin{equation*}
\check\nu^a(y)=\nu(a^{-1}y^{-1}a)\chi(a)^{-1}.
\end{equation*}
Conversely the Poisson integral of any $f\in L^\infty(N)$ is a bounded $\mathcal L_\alpha$-harmonic function.

For $t\in\R^+$ and $\wp\in A^+,$ let
\begin{equation*}
\delta_t^\wp=\Ad((\log t)\wp)|_N.
\end{equation*}
Then $t\mapsto\delta_t^\wp$ is a one parameter group of automorphisms of $N$ for which the corresponding eigenvalues on $\mathfrak{n}$ are all positive. It is known \cite{FS} that then $N$ has $\delta_t^\wp$-homogeneous norm: a non-negative continuous function $|\cdot|_\wp$ on $N$ such that $|n|_\wp=0$ if and only if $n=e$ and
\begin{equation*}
|\delta_t^\rho x|_\wp=t|x|_\wp.
\end{equation*}

For many years the best pointwise estimate in higher rank available in the literature was
\begin{equation*}
\nu(x)\leq C_\wp(1 + |x|_\wp)^{-\varepsilon}
\end{equation*}
for some $\varepsilon> 0,$ where $\wp\in A^+$ (\cite{D,DH}). These results, however, provide no way of determining  $\ve$. (For estimates for the Poisson kernel and its derivatives on rank-one $NA$ groups, i.e., $\dim A=1,$ see \cite{DHZ,DHU,U,BDH, DHcm}.)

A formula for determining an appropriate value of $\ve$ was provided by the authors in \cite{PUcm,PUpota}, although it is clear that the value of $\ve$ produced is far from best possible. Assume that the rank (dimension of $A$) is $k>1.$ Let $\nu$ be the Poisson kernel for the operator $\mathcal L_\alpha$  with $\alpha\in A^+.$

To simplify our notation we write $\Lambda$ to denote the set of roots
\begin{equation*}
\Lambda=\Xi\cup\Theta,
\end{equation*}
where
\begin{equation*}
\begin{split}
\Xi=&\{\xi_1,\ldots,\xi_m\},\\
\Theta=&\{\vartheta_1,\ldots,\vartheta_n\}.\\
\end{split}
\end{equation*}
For $\Lam_o\subset\Lam$ and $a\in A^+$ we set
\blab{DefOfgamLam}
\ali
\g_{\Lam_o}(a)&=\min_{\lam\in\Lam_o}\lam(a),\\
\o\g_{\Lam_o}(a)&=\min_{\lam\in\Lam_o}\fn{\lam(a)}{\lam^2}.\\
\eal
\elab

In this setting our final main result is the following.
\begin{theorem}\label{newupper}
Let $\nu$ be the Poisson kernel for the operator $\mathcal L_\alpha,$ defined in \eqref{defofl2}, with $\alpha\in A^+.$
Under the above assumptions on $N,$ for every $\wp\in A^+$ and $\varepsilon>0$ there exists a constant $C=C_{\wp,\varepsilon}>0$
such that
\begin{equation*}
\nu(m,v)\leq \\C(1+|(m,v)|_\wp)^{-\gamma},
\end{equation*}
where
\begin{equation*}
\gamma=
\begin{cases}
\g_\Theta(\rho)\o\g_\Theta(\alpha)=:\gamma_1&\text{for $\|m\|<\varepsilon,$ $\|v\|\geq\varepsilon,$}\\
\g_\Lam(\rho)\o\g_\Lam(\alpha)=:\gamma_2&\text{for $\|m\|\geq\varepsilon,$ $\|v\|<\varepsilon,$}\\
\max\{\gamma_1,\gamma_2,\frac{1}{2}\g_\Theta(\rho)\o\g_\Theta(\alpha)+
\frac{1}{2}\g_\Lam(\rho)\o\g_\Lam(\alpha)\}&\text{for $\|m\|\geq\varepsilon,$ $\|v\|\geq\varepsilon.$}\\
\end{cases}
\end{equation*}
\end{theorem}

We provide an example in \S\ref{example} demonstrating that this result does in fact provide a sharper estimate that the estimates found in~\cite{PUcm} and in~\cite{PUpota}.

\subsection{Structure of the paper}
The outline of the rest of the paper is as follows:

In \S\ref{efobm} and \S\ref{Sectwotw} we recall some basic facts about exponential functionals of Brownian motion and some estimates for the joint distribution of the maximum of the absolute value of the Brownain motion on the time interval $[0,t]$ and its position at time $t.$ In \S\ref{SSAbel} we give a formula for the evolution kernels in the special case that the nilpotent group is $\R^n.$   In \S\ref{PKE} we recall the construction of the Poisson kernel $\nu$ on $N$ and its extension $\nu^a(x)$ to $N\times\R^k.$
In \S\ref{eonM} and \S\ref{eonV} we consider diffusions on $M$ and $V$ respectively.
Theorem\refn{DecPs2} is proved in \S\ref{MAG}. Our main results are proved in \S\ref{eonN} (Theorem~\ref{ubpsigma}), and  \S\ref{parpkernel}, (Theorem~\ref{newupper}).
In \S\ref{example} we compare the estimate from Theorem~~\ref{newupper} with our previous results from \cite{PUcm,PUpota}.
\section{Preliminaries}
\subsection{Exponential functionals of Brownian motion}\label{efobm}
Let $b(s),$ $s\geq0,$ be the Brownian motion on $\R$ staring from
$a\in\R$ and normalized so that
\begin{equation}\label{gd}
\e_af(b(s))=\int_{\R}f(x+a)\frac{1}{\sqrt{4\pi s}}e^{-x^2\slash 4s}dx.
\end{equation}
Hence $\e b(s)=a$ and $\Var b(s)=2s.$
\begin{note}
Our normalization of the Brownian motion $b(s)$ is different than that typically used by probabilists who tend to assume that $\Var b(s)=s.$
\end{note}

For $d>0$ and $\mu>0$ we define the following exponential functional
\begin{equation}\label{perp}
I_{d,\mu}=\int_0^\infty e^{d(b(s)-\mu s)}ds.
\end{equation}
Such functionals are called \textit{perpetual functionals} in financial mathematics and they play an important role there (see e.g. \cite{Du,Yor-skladanka}).
\begin{theorem}[Dufresne, \cite{Du}]\label{tD}
Let $b(0)=0.$ Then the functional $I_{2,\mu}$ is distributed as
$(4\gamma_{\mu\slash 2})^{-1},$ where $\gamma_{\mu\slash 2}$
denotes a gamma random variable with parameter $\mu\slash 2,$
i.e., $\gamma_{\mu\slash 2}$ has a density
$(1\slash\Gamma(\mu\slash
2))x^{\frac{\mu}{2}-1}e^{-x}1_{[0,+\infty)}(x).$
\end{theorem}
The proof of Dufresne's theorem
can be found in many places. See for example \cite{D-SGD,DHZ} or
the survey paper \cite{Yor} and the references therein.

The \textit{inverse gamma density} on $(0,+\infty),$ with respect to $dx,$ is defined by
\begin{equation*}
h_{\mu,\gamma}(x)=C_{\mu,\gamma}x^{-\mu-1}e^{-\frac{\gamma}{x}}1_{(0,+\infty)}(x),
\end{equation*}
where $C_{\mu,\gamma}$ is the normalizing constant such that $\int_0^\infty h_{\mu,\gamma}(x)dx=1.$

As a corollary of Theorem~\ref{tD}, by scaling the Brownian
motion and changing the variable, we get the following theorem.

\begin{theorem}\label{tDHZ}
Let $b(0)=a.$ Then
\begin{equation*}
\e_af(I_{d,\mu})=c_{d,\mu}e^{\mu a}\int_0^\infty f(x)x^{-\mu\slash
d}\exp\left(-\frac{e^{da}}{d^2x}\right)\frac{dx}{x}.
\end{equation*}
In particular, $I_{2,\mu}$ has the inverse gamma density
$h_{\mu\slash 2,1\slash 4}.$
\end{theorem}

We will also need the following lemma.
\begin{lemma}\label{lperp}
Let $\sigma(u)=b(u)-2\alpha u$ be the $k$-dimensional Brownian motion with a drift, $d>0,$ and let $\ell\in(\R^k)^*$ be such that $\ell(\alpha)>0.$ Then
\begin{equation*}
\e_af\left(\int_0^\infty e^{d\ell(\sigma(u))}du\right)=c_{d,\ell,\alpha}e^{\gamma\ell(a)}\int_0^\infty f(u)u^{-\gamma\slash d}\exp\left(-\frac{e^{d\ell(a)}}{2d^2\ell^2u}\right)\frac{du}{u},
\end{equation*}
where $\gamma=2\ell(\alpha)\slash\ell^2.$

In particular, the functional $\int_0^\infty e^{d\ell(b(u)-2\alpha u)}du$ has the inverse gamma density $h_{2\ell(\alpha)\slash (d\ell^2),1\slash(d^2\ell^2)}.$
\end{lemma}
\begin{proof}
It follows from Theorem~\ref{tDHZ}. See \cite[Lemma~5.4]{PUcm} for details.
\end{proof}

\subsection{Some probabilistic lemmas}\label{Sectwotw}
If $b(t)$ is the Brownian motion starting from $x\in\R$ then the corresponding Wiener measure on the space $ C([0,\infty),\R)$ is denoted by $\mathbf W_x.$
The following lemma follows from formula 1.1.4 on p.~125 in \cite{Borodin}.
\begin{lemma}\label{bmsup}
There exists a constant $c>0$ such that for all $x\leq y,$
\begin{equation*}
\mathbf W_x(\sup_{0<s<t}|b(s)|\geq y)\leq ce^{-(y-x)^2/ 4t}.
\end{equation*}
\end{lemma}
The following two equalities follows easily from the reflection principle for the Brownian motion \cite{KS}.
\begin{lemma}\label{BndBoth}
If $x>a>0,$ then
\begin{equation*}
\mathbf W_0\left(\sup_{u\in[0,t]}b(u)\geq a\text{ and }b(t)\leq x\right)=2\mathbf W_0(b(t)>a)-\mathbf W_0(b(t)>x),
\end{equation*}
whereas if $x<a$ with $a>0,$ then
\begin{equation*}
\mathbf W_0\left(\sup_{u\in[0,t]}b(u)\geq a\text{ and }b(t)\leq x\right)=\mathbf W_0(b(t)>2a-x).
\end{equation*}
\end{lemma}
Let
\begin{equation*}
\Phi(x)=\frac{1}{\sqrt{4\pi} }\int_{-\infty}^xe^{-u^2\slash 4}du.
\end{equation*}
\begin{lemma}\label{FiCases}
Let $t>0.$ For $a\geq 0$, $x,y\in\R$ with $x<y,$ let
\begin{align*}
R_1=&\{-a\leq x<y\leq a\},&R_2=&\{x<y<-a\},\\
R_3=&\{a<x<y\},&R_4=&\{0<x<a<y\}.
\end{align*}
Then
\begin{multline}\label{BndSup}
\mathbf W_0\left(\sup_{u\in[0,t]}|b(u)|\geq a\text{ and }b(t)\in[x,y]\right)\leq\\
\begin{cases}
2\Phi\left(\frac{2a-x}{\sqrt t}\right)-2\Phi\left(\frac{2a-y}{\sqrt t}\right)+2\Phi\left(\frac{2a+y}{\sqrt t}\right)-2\Phi\left(\frac{2a+x}{\sqrt t}\right),&\text{on }R_1,\\
2\Phi\left(\frac{2a-x}{\sqrt t}\right)-2\Phi\left(\frac{2a-y}{\sqrt t}\right)+\Phi\left(\frac{-x}{\sqrt t}\right)-\Phi\left(\frac{-y}{\sqrt t}\right),&\text{on }R_2,\\
\Phi\left(\frac{y}{\sqrt t}\right)-\Phi\left(\frac{x}{\sqrt t}\right)+2\Phi\left(\frac{2a+y}{\sqrt t}\right)-2\Phi\left(\frac{2a+x}{\sqrt t}\right),&\text{on }R_3,\\
2\left(1-\Phi\left(\frac{a}{\sqrt t}\right)\right)-\Phi\left(\frac{y}{\sqrt t}\right)-\Phi\left(\frac{2a-x}{\sqrt t}\right)+\Phi\left(\frac{2a+x}{\sqrt t}\right)-\Phi\left(\frac{2a+y}{\sqrt t}\right),&\text{on }R_4.
\end{cases}
\end{multline}
\end{lemma}
\begin{proof}
We use
\begin{multline*}
\mathbf W_0\left(\sup_{u\in[0,t]}|b(u)|\geq a\text{ and }b(t)\in[x,y]\right)\\
\leq \mathbf W_0\left(\sup_{u\in[0,t]}b(u)\geq a\text{ and }b(t)\in[x,y]\right)\\
+\mathbf W_0\left(\sup_{u\in[0,t]}-b(u)\geq a\text{ and }-b(t)\in[-y,-x]\right).
\end{multline*}
Then the bound \eqref{BndSup} on each set $R_i$ follows from Lemma~\ref{BndBoth} by an easy calculation. 
\end{proof}
\begin{corollary}\label{UnifEps}
Assume that $a>|n|+\delta$, $\delta>0$, $\varepsilon<1,$ and $0<\varepsilon\slash 2<\delta.$ Then
\begin{multline*}
\varepsilon^{-1}\mathbf W_0\left(\sup_{u\in[0,t]}|b(u)|\geq a\text{ and }b(t)\in[n-\varepsilon\slash 2,n+\varepsilon\slash 2]\right)\\
\leq\frac{1}{\sqrt{\pi t}}\left(e^{-(2a-n)^2\slash(4t)}+e^{-(2a+n)^2\slash(4t)}\right).
\end{multline*}
\end{corollary}
\begin{proof} Let $x=n-\frac{\varepsilon}{2}$ and $y=n+\frac{\varepsilon}{2}.$ Our hypotheses imply that $-a<x<y<a$.  In particular
\begin{equation*}
0<(2a-y)\slash\sqrt t<(2a-x)\slash \sqrt t\text{ and }
0<(2a+x)\slash\sqrt t<(2a+y)\slash\sqrt t.
\end{equation*}
Hence, from Lemma~\ref{FiCases},
\begin{multline*}
\varepsilon^{-1}\mathbf W_0\left(\sup_{u\in[0,t]}|b(u)|\geq a\text{ and }b(t)\in[n-\varepsilon\slash 2,n+\varepsilon\slash 2]\right)\\
\leq\frac{1}{\sqrt{\pi t}}\left(e^{-(2a-x)^2\slash(4t)}+e^{-(2a+y)^2\slash(4t)}\right)
\end{multline*}
proving the Corollary.
\end{proof}
\begin{corollary}\label{Largen}
Assume that $a\ge 0.$ Then
\begin{multline*}
\limsup_{\varepsilon\to 0}\frac{1}{\varepsilon}\mathbf W_0\left(\sup_{u\in[0,t]}|b(u)|\geq a\text{ and }b(t)\in[n-\varepsilon\slash 2,n+\varepsilon\slash 2]\right)\\\leq
\begin{cases}
\frac{2}{\sqrt{\pi t}}e^{-(2a-|n|)^2\slash(4t)}&|n|<a,\\
\frac{1}{\sqrt{4\pi t}}e^{-n^2\slash(4t)}&0\leq a\leq |n|.
\end{cases}
\end{multline*}
\end{corollary}
\begin{proof}
The first statement is immediate from Corollary~\ref{UnifEps}. For the second statement note that
\begin{multline*}
\mathbf W_0\left(\sup_{u\in[0,t]}b(t)\geq a\text{ and }b(t)\in[n-\varepsilon\slash 2,n+\varepsilon\slash 2]\right)\\\leq\mathbf W_0\left(b(t)\in[n-\varepsilon\slash 2,n+\varepsilon\slash 2]\right)\\
=\frac{1}{\sqrt{4\pi t}}\int_{n-\varepsilon\slash 2}^{n+\varepsilon\slash 2}e^{-u^2\slash(4t)}du
\end{multline*}
from which the lemma follows.
\end{proof}
\subsection{Evolution equation in $\R^n$}\label{SSAbel}
Let
\begin{equation}\label{DefOfa}
L^t=\frac{1}{2}\sum_{i,j=1}^na_{ij}(t)\partial_{i}\partial_j+\sum_{j=1}^n b_j(t)\partial_j
\end{equation}
be a differential operator on $C^\infty(\R^n),$ where $\partial_i=\partial_{x_i}$ and $a(t)=[a_{ij}(t)]$ is a symmetric, positive definite matrix and the $a_{ij}$ and $b_j$ belong to $ C([0,\infty),\R)$. For
$s\leq t$, let $P_{t,s}$ be the fundamental solution for $L=\partial_s+L^s$ which is defined by formulas \eqref{opL}, \refp{IntPsg}, \refp{SgP} and \refp{DefOfUst}\, where $\mathcal L^{\sg(s)}$ is replaced by $L^s$.
Let
$$
\ali
A_{ij}(s,t)&=\int_s^ta_{ij}(u)du\equiv A_{i,j}\\
B_j(s,t)&=\int_s^tb_j(u)du\equiv B_j.\\
\eal
$$

\begin{proposition}\label{RnPsig}  Let $A=[A_{ij}]$ and $B=(B_1,B_2,\dots,B_n)^t$.  Then
\begin{equation}\label{ptsrn}
\begin{split}
P_{t,s}(x)
&=(2\pi)^{-\fn n2}(\det A)^{-\frac{1}{2}}e^{-\frac{1}{2}(A^{-1}(x-B))\cdot(x-B)}.
\end{split}
\end{equation}
\end{proposition}

\begin{proof} For $f_o\in C_\infty(\R^n)\cap L^2(\R^n)$, we write
\begin{equation*}
f(x,t)=U(s,t)f_o(x)=f_o*P_{t,s}(x).
\end{equation*}
We note that
\begin{align*}
\partial_tf(x,t)&=L^tf(x,t),\quad t>s,\\
f(x,s)&=f_o(x).
\end{align*}
We solve the above equation using the Fourier transform. See \cite[Proposition~2.10]{JEE} for details.
\end{proof}

\section{Meta-abelian groups}\label{MAG}
Let the notation be as in \S \ref{introduction}. We consider a family of automorphisms $\{\Phi(a)\}_{a\in\R^k}$ of $\mathfrak n,$ that leaves $\mathfrak m$ and $\mathfrak v$ invariant.
We identify linear transformations on $\mathfrak n$ with $(m+n)\times(m+n)$ matrices, allowing us to write
\begin{equation*}
\Phi(a)=\begin{bmatrix}
S(a)&0\\0&T(a)
\end{bmatrix},
\end{equation*}
where
\begin{equation*}
\begin{split}
S(a)=&\diag\left[e^{\xi_1(a)},\ldots,e^{\xi_m(a)}\right],\\
T(a)=&\diag\left[e^{\vartheta_1(a)},\ldots,e^{\vartheta_n(a)}\right].
\end{split}
\end{equation*}
We denote the
diagonal entries
 of $S(a)$ and $T(a)$ by
\begin{equation*}
\begin{split}
s_i(a)=&e^{\xi_i(a)},\;i=1,\ldots,m,\\
t_j(a)=&e^{\vartheta_j(a)},\;j=1,\ldots,n.
\end{split}
\end{equation*}
Let $\sigma$ be a continuous function from $[0,+\infty)$ to $A=\R^k,$ and denote
\begin{equation}\label{ssigma}
\Phi^\sigma(t)=\Phi(\sigma(t)),\;S^\sigma(t)=S(\sigma(t)),\;T^\sigma(t)=T(\sigma(t)).
\end{equation}
For $Z\in\mathfrak n$ let
\begin{equation*}
Z(t)=\Phi^\sigma(t)Z.
\end{equation*}
For $v\in V$ let
\begin{equation}\label{LMv}
\mathcal L_M^{\sigma,v,t}=\sum_{j=1}^{m}\left(\Ad(v)Y_j(t)\right)^2.
\end{equation}
Then
\begin{equation*}
\mathcal L_N^{\sigma(t)}f(m,v)=\mathcal L_V^{\sigma,t}f(m,\cdot)\big|_v+\mathcal L_M^{\sigma,v,t}f(\cdot,v)\big|_m,\;\;t\in\R^+,
\end{equation*}
is a family of left invariant operators on $N$ depending on $t\in\R^+$. Our aim is to estimate the evolution kernel $P^\sigma_{t,s}$ for the time dependent operator
$\mathcal L_N^{\sigma(t)}.$
\subsection{Evolution on $M$}\label{eonM}
We choose coordinates $y_i$ for $M$ for which $Y_i$ corresponds to $\partial_i=\partial_{y_i},$ $1\leq i\leq m.$ Let $\eta\in C([0,\infty),V)$ and consider the evolution on $M$ generated by the time dependent
operator
\begin{equation*}
\mathcal L_M^{\sigma,\eta,t}=\sum_{j=1}^{m}\left(\Ad(\eta(t))Y_j(t)\right)^2,
\end{equation*}
where,
\begin{equation*}
Y_j(t)=\Phi^\sigma(t)Y_j.
\end{equation*}
Then
\begin{equation*}
\Ad(\eta(t))Y_j(t)=\Ad(\eta(t))\Phi^\sigma(t)Y_j=\sum_{k=1}^{m}\psi_{j,k}(t)Y_k,
\end{equation*}
and consequently,
\begin{equation*}
\begin{split}
\sum_{j=1}^{m}\left(\Ad(\eta(t))Y_j(t)\right)^2=&\sum_{k,l=1}^m\sum_{j=1}^{m}\psi_{k,j}(t)\psi_{l,j}(t)Y_kY_l\\
=&\sum_{k,l=1}^m(\psi(t)\psi(t)^*)_{kl}Y_kY_l,
\end{split}
\end{equation*}
where
$\psi(t)=[\psi_{i,j}(t)]$
is the matrix of the operator $\Ad(\eta(t))\Phi^\sigma(t)\rest{M}.$ Thus the matrix $[a_{ij}(t)]$ from \eqref{DefOfa} for the operator $\mathcal L_M^{\sigma,\eta,y}$ is
\begin{equation*}
[a_{ij}(t)]=2\psi(t)\psi(t)^*=2\Ad(\eta(t))\Psi^\sigma(t)\rest{M}\left(\Ad(\eta(t))\Psi^\sigma(t)\rest{M}M\right)^*.
\end{equation*}
It follows from Proposition~\ref{RnPsig} that the evolution kernel $P^{M,\sigma,\eta}_{t,s}$ for the operator $\mathcal L_M^{\sigma,\eta,t}$ is Gaussian, and in our notation, is given by
\begin{equation}\label{PMform1}
P^{M,\sigma,\eta}_{t,s}(m,m^\prime)=\mathcal D(A^{\sigma,\eta}_M(t,s))e^{-\mathcal B(A^{\sigma,\eta}_M(t,s))(m-m^\prime)},
\end{equation}
where $m,m^\prime\in M=\R^{\dim M},$ and $\mathcal D,\mathcal B$ are defined in \eqref{bede}.
We will need the following two lemmas:
\begin{lemma}\label{B(A)}
Let $A$ be a positive semi-definite symmetric matrix. Then
\begin{equation*}
\mathcal B(A)(x)\geq\frac{\|x\|^2}{2\|A\|},
\end{equation*}
where $\|A\|$ is the $\ell^2\to\ell^2$-operator norm.
\end{lemma}
\begin{proof}
See e.g. \cite[Lemma 4.1]{JEE}
\end{proof}\begin{lemma}\label{blockdet}
Let $M$ and $D$ be square matrices and let
\begin{equation*}
A=\begin{bmatrix}
M&B\\C&D
\end{bmatrix}.
\end{equation*}
If $\det M\not=0$ then $\det A=\det M\det(D-CM^{-1}B).$
\end{lemma}
\begin{proof}
See e.g. \cite{matrix}.
\end{proof}
Now we prove an upper bound on $\mathcal{D}(A^{\sigma,\eta}_M(s,t))$ that is independent of $\eta.$
For simplicity of notation we identify $M$, $V$, and, $N$ with $\mathfrak m$, $\mathfrak v$, and $\mathfrak n$ using the exponential map.

\begin{lemma}\label{C(A)}
There is a constant $C>0$ such that
\begin{equation*}
\mathcal{D}(A^{\sigma,\eta}_M(s,t))\leq C\left(\prod_{i=1}^{m}\int_s^ts_{i}^\sigma(u)^2du\right)^{-1\slash 2}=CA_{M,\Pi}^\sigma(s,t)^{-1\slash 2},
\end{equation*}
where $s^\sigma_{i}(t)$ are the entries of the diagonal matrix $S^\sigma(t)$ defined in \eqref{ssigma}.
\end{lemma}
\begin{proof}
We omit the $t$ and $\sigma$ dependence for the sake of simplicity. From the lower triangularity of the adjoint action of $\mathfrak n$, for $X\in \mathfrak n=N$,
\begin{equation*}
\ad_X=\begin{bmatrix}
X_o&0\\
v^\tr&0
\end{bmatrix},\;\;\Ad_X=e^{\ad_X}=\begin{bmatrix}
e^{X_o}&0\\
v(X)^\tr&1
\end{bmatrix},
\end{equation*}
where the $X_o$ is an $(m-1)\times(m-1)$-matrix and $v$ is an $(m-1)\times 1$-column vector.

Then
\begin{equation}\label{formula1}
\begin{split}
\Ad_XS=&e^{\ad_X}\begin{bmatrix}S_0&0\\
0&s_{m}\end{bmatrix}=\begin{bmatrix}e^{X_o}S_o&0\\
v(X)^t S_o&s_{m}\end{bmatrix}.
\end{split}
\end{equation}
Let
\begin{equation*}
F^t=v(X)^t S_o.
\end{equation*}
Thus
\begin{equation*}
\Ad_XS(\Ad_XS)^t=\begin{bmatrix}
e^{X_o}S_oS_o^t e^{X_o^t}&G\\
G^t&s_{m}^2+|F|^2\end{bmatrix},
\end{equation*}
where
\begin{equation*}
G=e^{X_o}S_oF=e^{X_o}S_oS_o^t v(X).
\end{equation*}
Hence,
\begin{equation*}
A_M^{\sigma,\eta}(s,t)=2\begin{bmatrix}
A_o&B\\
B^t&A+E
\end{bmatrix},
\end{equation*}
where
\begin{align*}
A_o=&\int_s^te^{X_o(u)}S_o(u)S_o(u)^t e^{X_o(u)^t}du,&
B=&\int_s^t G(u)du,\\
A=&\int_s^ts_{m}^2(u)du,&
E=&\int_s^t|F(u)|^2du.
\end{align*}
From Lemma~\ref{blockdet},
\begin{equation*}
\begin{split}
\det A^{\sigma,\eta}_M(s,t)=&(\det A_o)(A+E-B^t A_o^{-1}B)\\
=&(\det A_o)A+(\det A_o)(E-B^t A_o^{-1}B)\\
=&(\det A_o)A+\det\begin{bmatrix}
A_o&B\\B^t&E
\end{bmatrix}.
\end{split}
\end{equation*}
The determinant on the right is non-negative since it is the $s_{m}=0$ case of formula \eqref{formula1}. Hence,
\begin{equation*}
\det A^{\sigma,\eta}_M(s,t)\geq A(\det A_o).
\end{equation*}
Our result follows by induction.
\end{proof}
Now we estimate the operator norm of the matrix
\begin{equation}\label{DefOfAseta}
A^{\sigma,\eta}_M(0,t)=2\int_0^t[\Ad(\eta(u))S^\sigma(u)]\left[\Ad(\eta(u))S^\sigma(u)\right]^{t}du.
\end{equation}
\begin{lemma}\label{normaasigmaeta}
Let $\eta=\eta(u)=(\eta_1(u),\ldots,\eta_n(u))$ be a continuous function. Then there exists a constant $C>0$  such that
\begin{equation*}
\|A^{\sigma,\eta}_M(0,t)\|\leq C(1+\Lambda^\eta(0,t)^{2k_o})\sum_{j=1}^{m}\int_0^ts_{j}^\sigma(u)^2du,
\end{equation*}
where
\begin{equation*}
\Lambda^\eta(s,t)=\sup_{s\leq u\leq t}\|\eta(u)\|.
\end{equation*}
\end{lemma}
\begin{proof} (We recall that $\|\cdot\|$ denotes the $\ell^2$-norm.) We note first that for $X\in \mf{n},$
\blab{EstadX}
\ali
\Ad_X\rest{\mf m}&=\sum_{j=0}^{k_o}\fn{\lp\ad_X\rest{\mf m}\rp^j}{j!},\\
\|\Ad_X\rest{\mf m}\|&\le C(1+\| \ad_X\|)^{k_o}\\
&\le C'(1+\| X\|)^{k_o}.\\
\eal
\elab
Our result follows by bringing the norm inside the integral in\refp{DefOfAseta}.
\end{proof}

\subsection{Evolution on $V$}\label{eonV}
Recall that we identified $V$ with $\R^n.$ The matrix $T^\sigma(t)=\Phi^\sigma(t)\big|_V$ is of the form
\begin{equation*}
T^\sigma(t)=\diag\left[
e^{\vartheta_1(\sigma(t))},\ldots,e^{\vartheta_n(\sigma(t))}\right],
\end{equation*}
where, $\vartheta_1,\ldots,\vartheta_n\in(\R^n)^*.$
Now we consider the evolution process $\eta(t)$ on $V$ generated by
\begin{equation*}
\mathcal L_V^{\sigma,t}=\sum_{j=1}^{n}X_j(t)^2=\sum_{j=1}^{n}(T^\sigma(t)X_j)^2
\end{equation*}
(see the notation introduced in \eqref{operators} on p.~\pageref{operators}).
Thus, since $X_j=\partial_{v_j},$
\begin{equation*}
\mathcal L_V^{\sigma,t}=\sum_{j=1}^{n}e^{2\vartheta_j(\sigma(t))}\partial_{v_j}^2.
\end{equation*}
The matrix $a(t)=[a_{ij}(t)],$ defined in \eqref{DefOfa}, for $\mathcal L_V^{\sigma,t}$ is equal to
\begin{equation*}
a^\sigma_V(t)=2T^\sigma(t)T^\sigma(t)^*=2\diag\left[e^{2\vartheta_1(\sigma(t))},\ldots,e^{2\vartheta_{n}(\sigma(t))}\right].
\end{equation*}
Let $b(t)$ be the $1$-dimensional Brownian motion normalized so that
\begin{equation*}
\mathbf{W}_x(b(t)\in dy)=p_t(x,dy)=\frac{1}{(4\pi t)^{1\slash 2}}e^{-(x-y)^2\slash 4t}dy.
\end{equation*}
Then, by \eqref{ptsrn},
\begin{equation}\label{KVsigma}
P^{V,\sigma}_{s,t}(x,dz)=\prod_{1\leq j\leq n}p_{A^\sg_{V,i}(s,t)}(x_j,dz_j).
\end{equation}

Thus the process $\eta(t)$ generated by $\mathcal L_V^{\sigma,t}$ has coordinates $\eta_j(t)$ which are independent Brownian motions with time changed according to the clock governed by $\sigma.$
Let
\begin{equation}\label{At}
A^\sigma_V(0,t)=\int_0^ta^\sigma_V(u)du.
\end{equation}
Since $A^\sigma_V(0,t)$ is diagonal we see
\begin{equation}\label{detAt}
\ali
(\det A^\sigma_V(0,t))^{-1\slash 2}&=\left(\prod_{j=1}^{n}\int_0^te^{2\vartheta_j(\sigma(u))}du\right)^{-1\slash 2},\\
\|A^\sigma_V(0,t)\|&\leq \sum_{j=1}^n\int_0^te^{2\vartheta_j(\sigma(u))}du\\
&=A_{V,\Sigma}^\sigma(0,t).\\
\eal
\end{equation}

\subsection{Proof of Theorem\refn{DecPs2}}\label{PfOfdecPsg}
Theorem\refn{DecPs2} follows from formula\refp{PMform1} together with \cite[Corollary~3.5]{JEE} and formula \cite[(3.1)]{JEE} with $n=1$.

\section{Estimate for the evolution on $N$}\label{eonN}
In this section we estimate the evolution kernel on $N=M\rtimes V.$ Denote
\begin{equation*}
P^\sigma_{t,s}(m,v):=P^\sigma_{t,s}(0,0;m,v).
\end{equation*}
The main result of this section is the following estimate where $k_o$ is as in\refp{DefOfko}.
\begin{theorem}\label{PreEst} There are positive constants $C,D$  such that
\begin{multline*}
A_{M,\Pi}^{\sigma}(0,t)^{\frac{1}{2}}A_{V,\Pi}^{\sigma}(0,t)^{\frac{1}{2}}P^\sigma_{t,s}(m,v)\leq\\
C(\|m\|^{\frac{1}{2k_o}}+1)
\exp\left(-\frac{D\|v\|^ 2}{A_{V,\Sigma}^\sigma(0,t)}-\frac{D\|m\|^2}{(\|m\|^{\frac {1}{2k_o}}+\|v\| +2)^{2k_o}A_{M,\Sigma}^\sigma(0,t)}\right)\\
+CA_{V,\Sigma}^\sigma(0,t)^{1\slash 2}\exp\left(-D\frac{\|m\|^{\frac{1}{k_o}}+\|v\|^ 2}{A_{V,\Sigma}^\sigma(0,t)}\right).
\end{multline*}
\end{theorem}
\begin{proof}
We allow the constants $C$ and $D$ to change from line to line. By Lemma~\ref{B(A)} and Lemma~\ref{C(A)},
\begin{equation}\label{estkse}
\begin{split}
P^{M,\sigma,\eta}_{t,s}(m,m^\prime)=&\mathcal D(A^{\sigma,\eta}_M(s,t))e^{-\mathcal B(A^{\sigma,\eta}_M(s,t))(m-m^\prime)}\\
\leq&CA_{M,\Pi}^\sigma(s,t)^{-1\slash 2}e^{-\frac{\|m-m^\prime\|^2}{2\|A^{\sigma,\eta}_M(s,t)\|}}.
\end{split}
\end{equation}

From Theorem\refn{DecPs2}, for $m,m^\prime\in M$ and $v,v^\prime\in V$,
\begin{equation*}\label{Ksigma}
\begin{split}
\int_V P^\sigma_{t,0}&(m,v;m^\prime,v^\prime)\psi(v^\prime)dv^\prime=\int P^{M,\sigma,\eta}_{t,0}(m,m^\prime)\psi(\eta(t))\, d\mathbf{W}^{V,\sigma}_{v}(\eta)\\
\leq&CA_{M,\Pi}^\sigma(0,t)^{-1\slash 2}\int\psi(\eta(t))e^{-\frac{\|m-m^\prime\|^2}{2\|A^{\sigma,\eta}_M(0,t)\|}}d\mathbf{W}^{V,\sigma}_{v}(\eta).
\end{split}
\end{equation*}
Then, by Lemma~\ref{normaasigmaeta},
\begin{multline}\label{integral}
A_{M,\Pi}^\sigma(0,t)^{1\slash 2}\int_V P^\sigma_{t,0}(m,v)\psi(v)dv\\
\leq C\int\exp\left(-\frac{D\|m\|^2}{(1+\Lambda^\eta(0,t)^{2k_o})A_{M,\Sigma}^\sigma(0,t)}\right)\psi(\eta(t))\, d\mathbf{W}^{V,\sigma}_0(\eta).
\end{multline}
For  $v\in\R^n$ given and $\varepsilon>0$, let
\begin{equation*}
\psi_\varepsilon(\cdot)=\varepsilon^{-n}\one_{B_\varepsilon(v)}(\cdot),
\end{equation*}
where
\begin{equation*}
B_\varepsilon(v)=\prod_{j=1}^nB_\varepsilon^1(v_j)\text{ and }B_\varepsilon^1(v_j)=[v_j-\varepsilon\slash 2,v_j+\varepsilon\slash 2].
\end{equation*}
We will estimate\refp{integral} with $\psi_\varepsilon$ in place of $\psi$ as $\varepsilon$ tends to zero.

Let $\e_v^\eta$ denote expectation with respect to $d\mathbf{W}_v^{V,\sigma}(\eta)$. For $k=1,2,\dotsc,$ define the sets of paths in $V,$
\begin{equation*}
\mathcal{A}_k(t)=\{\eta:k-1\leq\Lambda^\eta(0,t)=\sup_{0\leq u\leq t}\|\eta(u)\|_{\infty}<k\},
\end{equation*}
where by $\|\cdot\|_\infty$ we denote the maximum norm $\|x\|_\infty=\max_{1\leq i\leq n}|x_i|.$
The integral on the right in \eqref{integral} can be written as an infinite sum and estimated as follows
\begin{equation}\label{ExpIneq}
\begin{split}
\sum_{k=1}^\infty\e_0^\eta\exp\left(-\frac{D\|m\|^2}{(1+\Lambda^\eta(0,t)^{2k_o})
A_{M,\Sigma}^\sigma(0,t)}\right)\psi_\varepsilon(\eta(t))\one_{\mathcal{A}_k(t)}(\eta)\\
\leq\sum_{k=1}^\infty\exp\left(-\frac{D\|m\|^2}{k^{2k_o}
A_{M,\Sigma}^\sigma(0,t)}\right)\e_0^\eta\psi_\varepsilon(\eta(t))\one_{\mathcal{A}_k(t)}(\eta).
\end{split}
\end{equation}
To simplify notation we introduce
\begin{equation*}
\begin{split}
c_k&=\exp\left(-\frac{D\|m\|^2}{k^{2k_o}A_{M,\Sigma}^\sigma(0,t)}\right),\\
\mathcal{E}_k(\varepsilon)&=\e_0^\eta\psi_\varepsilon(\eta(t))\one_{\mathcal{A}_k(t)}(\eta)
=\varepsilon^{-n}\mathbf{W}^{V,\sigma}_0\left(\eta\in\mathcal{A}_k(t)\text{ and }\eta(t)\in B_\varepsilon(v)\right).
\end{split}
\end{equation*}
Let $v\not=0$ and choose $\varepsilon\slash 2<\|v\|_\infty$. If $\eta\in\mathcal{A}_k(t)$ and $\eta(t)\in B_\varepsilon(v)$ then $\|\eta(t)\|_\infty\geq\|v\|_\infty-\varepsilon\slash 2$. Hence,
\begin{equation}\label{zero}
\mathcal{E}_k=0\text{ for }k<\|v\|_\infty-\varepsilon\slash 2.
\end{equation}
Let, for $k=1,2,\dotsc,$
\begin{equation*}
\Lambda^{\eta_j}(0,t)=\sup_{0\leq u\leq t}|\eta_j(u)|\text{ and }\mathcal{A}_k^j(t)=\{\eta:k-1\leq\Lambda^{\eta_j}(0,t)<k\}.
\end{equation*}
Since the coordinates $\eta_j(t)$ of $\eta(t)$ are independent, we can estimate
\begin{equation}\label{ek}
\begin{split}
&\mathcal{E}_k(\varepsilon)=\varepsilon^{-n}\mathbf{W}^{V,\sigma}_0\left(\eta\in\mathcal{A}_k(t)\wedge\eta(t)\in B_\varepsilon(v)\right)\\
\leq&\varepsilon^{-n}\sum_{j=1}^n\mathbf{W}^{V,\sigma}_0\left(\eta\in\mathcal{A}_k^j(t)\wedge\eta(t)\in B_\varepsilon(v)\right)\\
=&\varepsilon^{-n}\sum_{j=1}^n\mathbf{W}^{V,\sigma}_0\left(\eta\in\mathcal{A}_k^j(t)\wedge\eta_j(t)\in B^1_\varepsilon(v_j)\right)\mathbf{W}^{V,\sigma}_0\left(\forall i\not=j, \eta_i(t)\in B^1_\varepsilon(v_i)\right)\\
=&\sum_{j=1}^n\varepsilon^{-1}\mathbf{W}_0\left(\eta\in\mathcal{A}_k^j(A_{V,j}^\sigma(0,t))\wedge\eta_j(A_{1,j}^\sigma(0,t))\in B^1_\varepsilon(v_j)\right)\times\\
&\times\prod_{i\not=j}\left(\varepsilon^{-1}\mathbf{W}_0\left(\eta_i(A_{V,i}^\sigma(0,t))\in B^1_\varepsilon(v_i)\right)\right).
\end{split}
\end{equation}
\begin{lemma}\label{cor4}
Assume that $a>\|v\|_\infty+\delta,$ $\delta>0,$ and $0<\varepsilon\slash 2<\delta.$ Then
\begin{multline*}
\varepsilon^{-n}\mathbf{W}_0^{V,\sigma}(\sup_{u\in[0,t]}\|\eta(u)\|_\infty\geq a\text{ and }\eta(t)\in B_\varepsilon(v))\\
\leq A_{V,\Pi}^\sigma(0,t)^{-1\slash 2}\sum_{j=1}^n\left(e^{-(2a-v_j)^2\slash 4A_{V,j}^\sigma(0,t)}+e^{-(2a+v_j)^2\slash 4A_{V,j}^\sigma(0,t)}\right).
\end{multline*}
\end{lemma}
\begin{proof}
Reasoning as in \eqref{ek} we see that the left side of the above inequality is bounded by
\begin{multline*}
C\sum_{j=1}^n\left(\prod_{i\not=j}A_{V,i}^\sigma(0,t)\right)^{-1\slash 2}\\\times\varepsilon^{-1}\mathbf{W}_0\left(\sup_{u\in[0,A_{V,j}^\sigma(0,t)]}|\eta_j(u)|\geq a\text{ and }\eta_j(A_{V,j}^\sigma(0,t))\in B^1_\varepsilon(v_j)\right).
\end{multline*}
By our assumption it follows that for every $j,$ $a>|v_j|+\delta.$ Hence, the result follows by Corollary~\ref{UnifEps}.
\end{proof}
\begin{lemma}\label{l.4.3}
We have
\begin{equation*}
A_{M,\Pi}^{\sigma}(0,t)^{\frac{1}{2}}P^\sigma_{t,0}(m,v)\le CI,
\end{equation*}
where
\begin{equation*}
I=\limsup_{\varepsilon\to 0^+}\sum_{k\geq\|v\|_\infty}c_k\mathcal{E}_k(\varepsilon).
\end{equation*}
Furthermore, the sum converges uniformly in $\varepsilon$.
\end{lemma}
\begin{proof}
The inequality follows by letting $\varepsilon$ tend to 0 in \eqref{ExpIneq}. The uniform convergence follows from Lemma~\ref{cor4}.
\end{proof}
Let $n_o$ be the smallest natural number such that $n_o\geq\|v\|_\infty$.
\begin{lemma}\label{enkest}
We have the following estimates
\begin{equation*}
\limsup_{\varepsilon\to 0^+}\mathcal{E}_{n_o}(\varepsilon)\leq CA_{V,\Pi}^\sigma(0,t)^{-1\slash 2}e^{-\|v\|^2_\infty\slash 4A_{V,\Sigma}^\sigma(0,t)},
\end{equation*}
while for $k\geq n_o+1,$
\begin{equation*}
\limsup_{\varepsilon\to 0^+}\mathcal{E}_k(\varepsilon)\leq CA_{V,\Pi}^\sigma(0,t)^{-1\slash 2}\exp\left(-\frac{(2(k-1)-\|v\|_\infty)^2}{4A_{V,\Sigma}^\sigma(0,t)}\right).
\end{equation*}
\end{lemma}
\begin{proof}
Consider $\mathcal{E}_{n_o}.$ Let $j\in\{1,\ldots,n\}$ be fixed. Suppose first that $|v_j|<n_o-1.$ Then, using Corollary~\refp{Largen}, the $j$-th term in \eqref{ek} (with $k=n_o$) is dominated by a multiple of
\begin{equation*}
A_{V,\Pi}^\sigma(0,t)^{-1\slash 2}e^{-\frac{(2(n_o-1)-|v_j|)^2}{4A_{V,j}^\sigma(0,t)}}\prod_{i\not=j}e^{-\frac{|v_i|^2}{4A_{V,i}^\sigma(0,t)}}.
\end{equation*}
Notice that $|v_j|$ cannot be equal to $\|v\|_\infty.$ Thus we are done in this case. Now suppose that
$|v_j|\geq n_o-1.$ Then, using Corollary~\ref{Largen} again, we dominate the $j$-th term in \eqref{ek} by a multiple of
\begin{equation*}
A_{V,\Pi}^\sigma(0,t)^{-1\slash 2}e^{-\frac{|v_j|^2}{4A_{V,j}^\sigma(0,t)}}\prod_{i\not=j}e^{-\frac{|v_i|^2}{4A_{V,i}^\sigma(0,t)}}.
\end{equation*}
The result for $\mathcal{E}_{n_o}$ follows.

Now we consider $\mathcal{E}_k.$ Since $k\geq n_o+1$ it follows that $k-1\geq |v_j|$ for every $j.$ Therefore, by Corollary~\ref{Largen} the $j$-th term in \eqref{ek} is estimated by
\begin{equation*}
CA_{V,\Pi}^\sigma(0,t)^{-1\slash 2}e^{-\frac{(2(k-1)-|v_j|)^2}{4A_{V,j}^\sigma(0,t)}}\prod_{i\not=j}e^{-\frac{|v_i|^2}{4A_{V,i}^\sigma(0,t)}}.
\end{equation*}
\end{proof}

Next we estimate $I=\limsup_{\varepsilon\to 0^+}\sum_{k\geq\|v\|_\infty}c_k\mathcal{E}_k(\varepsilon).$
From Lemma~\ref{enkest}
\begin{equation}\label{first}
\begin{split}
A_{V,\Pi}^\sigma(0,t)^{1\slash 2}I
=&A_{V,\Pi}^\sigma(0,t)^{1\slash 2}\limsup_{\varepsilon\to 0^+}\left(c_{n_o}\mathcal{E}_{n_o}(\varepsilon)+\sum_{k\geq n_o+1}c_k\mathcal{E}_k(\varepsilon)\right)\\
\leq&C\exp\left(-\frac{\|v\|_\infty^2}{2A_{V,\Sigma}^\sigma(0,t)}-\frac{D\|m\|^2}{n_o^{2k_o }A_{M,\Sigma}^\sigma(0,t)
}\right)\\
&+\sum_{k=n_o+1}^\infty
\exp\left(-\frac{D\|m\|^2}{k^{2k_o }A_{M,\Sigma}^\sigma(0,t)}-\frac{(2(k-1)-\|v\|_\infty)^2}{2A_{V,\Sigma}^\sigma(0,t)}\right).
\end{split}
\end{equation}

For $a,b$  non-negative $
a+b\ge \sqrt{a^2+b^2}.
$
Also, for $k\ge n_o+1$,
\begin{equation*}
\begin{split}
(k-1)+(k-1)-\|v\|_\infty&\geq n_o+(k-1-\|v\|_\infty),\\
k-1-\|v\|_\infty&\geq n_o-\|v\|_\infty\geq 0.
\end{split}
\end{equation*}
Hence the summation in the last line of \eqref{first} is bounded by
\begin{multline}\label{second}
\sum_{k=n_o+1}^\infty\exp\left(-\frac{D\|m\|^2}{k^{2k_o }A_{M,\Sigma}^\sigma(0,t)}-
\frac{(n_o+(k-1)-\|v\|_\infty)^2}{4A_{V,\Sigma}^\sigma(0,t)}\right)\\
\leq e^{-\frac{n_o^2}{4A_{V,\Sigma}^\sigma(0,t)}}
\sum_{k=n_o+1}^\infty
\exp\left(-\frac{D\|m\|^2}{k^{2k_o }A_{M,\Sigma}^\sigma(0,t)}-
\frac{(k-1-\|v\|_\infty)^2}{4A_{V,\Sigma}^\sigma(0,t)}\right).
\end{multline}
We split the sum in \eqref{second} into two parts: $n_o+1\leq k\leq n_o+\|m\|^{1\slash (2k_o) }$ and $k>n_o+\|m\|^{1\slash (2k_o) },$ and estimate the corresponding parts by the following two terms.
\begin{equation*}
\|m\|^{\frac{1}{2k_o }}e^{-\frac{n_o^2}{4A_{V,\Sigma}^\sigma(0,t)}}\exp
\left(-\frac{D\|m\|^2}{(\|m\|^\frac{1}{2k_o }+\|v\|_\infty+2)^{2k_o }A_{M,\Sigma}^\sigma(0,t)}\right)
\end{equation*}
and
\begin{equation*}
e^{-\frac{n_o^2}{4A_{V,\Sigma}^\sigma(0,t)}}\sum_{k\geq n_o+\|m\|^\frac{1}{2k_o }+1}\exp\left(-\frac{D\|m\|^2}{k^2A_{M,\Sigma}^\sigma(0,t)}-
\frac{(k-1-n_o)^2}{4A_{V,\Sigma}^\sigma(0,t)}\right).
\end{equation*}
The above expression is bounded by
\begin{equation*}
e^{-\frac{n_o^2}{4A_{V,\Sigma}^\sigma(0,t)}}\int_{\|m\|^\frac{1}{2k_o }}^\infty e^{-\frac{r^2}{4A_{V,\Sigma}^\sigma(0,t)}}dr
\leq 2A_{V,\Sigma}^\sigma(0,t)^{1\slash 2}e^{-\frac{\|v\|_\infty^2}{4A_{V,\Sigma}^\sigma(0,t)}-\frac{\|m\|^{\frac{1}{k_o }}}{4A_{V,\Sigma}^\sigma(0,t)}}.
\end{equation*}
Hence, by Lemma~\ref{l.4.3} and Lemma~\ref{enkest}, Theorem\refp{PreEst} follows.
\end{proof}
\subsection{Proof of Theorem~\ref{ubpsigma}}
To simplify our notation we set
\begin{equation}\label{Ajsigma}
\begin{split}
A_0&=A_{N,\Pi}^\sigma(0,t)^{1\slash 2}=A_{M,\Pi}^\sigma(0,t)^{1\slash 2}A_{V,\Pi}^\sigma(0,t)^{1\slash 2},\\
A_1&=A_{V,\Sigma}^\sigma(0,t),\\
A_2&=A_{M,\Sigma}^{\sigma}(0,t),\\
A_3&=A_{N,\Sigma}^{\sigma}(0,t).
\end{split}
\end{equation}
For $k\in\N$ and the $\ell^2$-norm $\|\cdot\|,$ we let
\begin{equation*}
\phi_k(m)=\left(\frac{\|m\|^{1\slash k}}{\|m\|^{1\slash k}+1}\right)^k,\;\;m\in M.
\end{equation*}

It follows from Theorem~\ref{PreEst} that
there are positive constants $C$ and $D$ such that in the region $\|v\|\leq\|m\|^{\frac{1}{2k_o}}$,
\begin{multline}\label{vmniejsze}
A_0P^\sigma_{t,0}(m,v)\\
\leq C(\|m\|^{1\slash (2k_o)}+1)\exp\left(-D\frac{\|v\|^2}{A_1}-D\frac{\|m\|}{A_2}\phi_{2k_o}(m)\right)\\+CA_1^{1\slash 2}\exp\left(-D\frac{\|m\|^{1\slash k_o}+\|v\|^2}{A_1}\right).
\end{multline}
In fact, if $\|v\|\leq\|m\|^{\frac{1}{2k_o}}$ then
\begin{equation*}
\frac{\|m\|^2}{(\|m\|^{1\slash(2k_o)}+\|v\|+2)^{2k_o}}\geq\frac{\|m\|^2}{2^{2k_o}(\|m\|^{1\slash(2k_o)}+1)^{2k_o}}
=\frac{1}{2^{2k_o}}\phi_{2k_o}(m)\|m\|,
\end{equation*}
and we get \eqref{vmniejsze}.

Now we consider $\|v\|\geq \|m\|^{\frac{1}{2k_o}}.$ Again, it follows from Theorem~\ref{PreEst} that there are positive constants $C,D$ such that
\begin{equation}\label{vwieksze}
A_0P^\sigma_{t,0}(m,v)\leq C(\|m\|^{1\slash (2k_o)}+1+A_1^{1\slash 2})\exp\left(-D\frac{\|m\|^{1\slash k_o}+\|v\|^2}{A_1}\right).
\end{equation}
In fact, it is easy to see that if $\|v\|\geq\|m\|^{1\slash(2k_o)}$ then there is a constant $D>0$ such that
\begin{equation*}
\frac{\|v\|^2}{A_1}+\frac{\|m\|^2}{(\|m\|^{1\slash(2k_o)}+\|v\|+2)^{2k_o}A_2}\geq D\frac{\|m\|^{1\slash k_0}+\|v\|^2}{A_1}.
\end{equation*}
Hence, \eqref{vwieksze} follows from Theorem~\ref{PreEst}.

Since $0<\phi_{2k_o}(m)\leq 1$ for all $m,$ we note that
\begin{equation*}
\|m\|\phi_{2k_o}(m)\geq\|m\|^{1\slash k_o}\text{ for }\|m\|\geq 1
\end{equation*}
and
\begin{equation*}
\|m\|^{1\slash k_o}\geq\|m\|\phi_{2k_o}(m)\geq C\|m\|^2\text{ for }\|m\|<1.
\end{equation*}
Since $A_1\leq A_3$ and $A_2\leq A_3,$
Theorem~\ref{ubpsigma} follows from \eqref{vmniejsze} and \eqref{vwieksze}.
\section{Upper estimate for the Poisson kernel}\label{parpkernel}

\subsection{Poisson kernel}\label{PKE}
 Let $\mu_t$ be the semigroup of probability measures on $S=N\rtimes\R^k$ generated by $\mathcal{L}_\alpha.$ It is known \cite{DH,DHZ} that
\begin{equation*}
\lim_{t\to\infty}(\pi(\check\mu_t),f)=(\nu,f),
\end{equation*}
where $\pi$ denotes the projection from $S$ onto $N$ and $(\check\mu,f)=(\mu,\check f),$ $\check f(x)=f(x^{-1}).$
Let $a\in\R^k$
and let $\mu$ be a measure on $N.$ We define
\begin{equation*}
(\mu^a,f)=(\mu,f\circ\Ad(a)).
\end{equation*}
For $a\in\R^k$ we have
\begin{equation}\label{defnua}
\nu^a(x)=\nu(a^{-1}xa)\chi(a)^{-1},\;\;x\in N,
\end{equation}
where $\chi$  is as in \refp{DefOfX}.

It is an easy calculation to check that
\begin{equation}\label{p1}
\lim_{t\to\infty}(\pi(\check\mu_t)^a,f)=(\nu^a,f).
\end{equation}
We need the following
\begin{lemma}\label{p2}
We have
\begin{equation*}
(\pi(\check\mu_t)^a,f)=(\e_a\check P^\sigma_{t,0},f).
\end{equation*}
\end{lemma}
\begin{proof}
This equality follows from  formula\refp{probform}. See \cite[Lemma~4.1]{PUcm} for the details.
\end{proof}
By \eqref{p1} and Lemma~\ref{p2} it follows that
\begin{equation}\label{r}
(\nu^a,f)=\lim_{t\to\infty}(\pi(\check\mu_t)^a,f)
=\lim_{t\to\infty}(\e_a\check P^\sigma_{t,0},f).
\end{equation}

Let $\nu(x)=\nu(m,v),$ $m\in\R^{m},$ $v\in\R^n,$ be the Poisson kernel on $N$ for the operator $\mathcal L_\alpha$ in \eqref{defofl2}.
Recall that we assume that
\begin{equation*}
\lambda(\alpha)>0\text{ for all }\lambda\in\Lambda.
\end{equation*}
Hence $\alpha$ belongs to the positive Weyl chamber $A^+.$ The operator $\Delta_\alpha$ generates the Brownian motion with drift $-2\alpha,$
\begin{equation*}
\sigma(u)=b(u)-2\alpha u,
\end{equation*}
where $b(u)$ is the $k$-dimensional standard Brownian motion normalized so that $\Var b_u=2u.$

Let $\nu^a$ be as in \eqref{defnua}. We also use the notation introduced in\refp{DefOfgamLam}.
\begin{theorem}\label{thnua}
For all compact subsets $K\not\ni e$ of $N=M\rtimes V,$ all $\wp\in A^+,$ and all $\varepsilon>0$ there exists a constant $C=C(K,\wp,\varepsilon)>0$ such that for all $s<0,$
\begin{multline}\label{estth61}
\nu^{s\wp}(x)\leq Ce^{-\rho_0(s\wp)}e^{(s\slash 2)\g_\Theta(\wp)\o\g_\Theta(\alpha)}
e^{(s\slash 2)\g_\Lam(\wp)\o\g_\Lam(\alpha)}\\
\text{ if }x=(m,v)\in K_1=K\cap\{\|m\|\geq\varepsilon,\;\|v\|\geq\varepsilon\},
\end{multline}
\begin{multline}\label{estth611}
\nu^{s\wp}(x)\leq Ce^{-\rho_0(s\wp)}e^{s\g_\Theta(\wp)\o\g_\Theta(\alpha)}\\
\text{ if }x=(m,v)\in K_2=K\cap\{\|m\|<\varepsilon,\;\|v\|\geq\varepsilon\},
\end{multline}
and
\begin{multline}\label{estth6111}
\nu^{s\wp}(x)\leq Ce^{-\rho_0(s\wp)}e^{s\g_\Lam(\wp)\o\g_\Lam(\alpha)}\\\text{ if }x=(m,v)\in K_3=K\cap\{\|m\|\geq\varepsilon,\;\|v\|<\varepsilon\}.
\end{multline}
\end{theorem}
\begin{proof}
First we consider elements $x=(m,v)$ from the set $K_1.$
Let $A_j$ be defined as in \eqref{Ajsigma}  but with $t=\infty.$ By Theorem~\ref{ubpsigma}, we have
\begin{equation}\label{nu*0}
\nu^{s\wp}\leq C\e_{s\wp}A_0^{-1}\exp\left(-\frac{D}{A_1}-\frac{D}{A_3}\right)+C\e_{s\wp} A_0^{-1}A_1^{1\slash 2}\exp\left(-\frac{D}{ A_1}-\frac{D}{ A_3}\right).
\end{equation}
We estimate the first expectation on the right.
\begin{equation}\label{*1}
\begin{split}
\e_{s\wp}&A_0^{-1}\exp\left(-\frac{D}{ A_1}-\frac{D}{ A_3}\right)\\
\leq&\left(\e_{s\wp}( A_0^{-1})^2\right)^{1\slash 2}\left(\e_{s\wp}\exp\left(-\frac{2D}{ A_1}-\frac{2D}{ A_3}\right)\right)^{1\slash 2}\\
\leq&\left(\e_{s\wp}( A_0^{-1})^2\right)^{1\slash 2}\left(\e_{s\wp}\exp\left(-\frac{4D}{ A_1}\right)\right)^{1\slash 4}\left(\e_{s\wp}\exp\left(-\frac{4D}{A_3}\right)\right)^{1\slash 4}.
\end{split}
\end{equation}
By the Cauchy-Schwarz inequality we get,
\begin{equation}\label{eqperp1}
\begin{split}
\e_{s\wp}( A_0^{-1})^2=&\e_{s\wp}(A_{M,\Pi}^\sigma)^{-1}(A_{V,\Pi}^{\sigma})^{-1}\\
=&e^{-2\rho_0(s\wp)}\e_{0}(A_{M,\Pi}^\sigma)^{-1}(A_{V,\Pi}^{\sigma})^{-1}\\
\leq&e^{-2\rho_0(s\wp)}(\e_{0}(A_{M,\Pi}^\sigma)^{-2})^{1\slash 2}(\e_{0}(A_{V,\Pi}^{\sigma})^{-2})^{1\slash 2}.
\end{split}
\end{equation}
Since, by Lemma~\ref{lperp}, the expected values $\e_0(A_{M,j}^\sigma)^{-d},$ $j=1,\ldots,m,$ and $\e_0(A_{V,i}^\sigma)^{-d},$ $i=1,\ldots, n,$ are finite
for all $d>0,$ we can apply the Cauchy-Schwarz inequality successively to each of the remaining expectation in \eqref{eqperp1} and conclude their finiteness.

Now we consider $\e_{s\wp}\exp(-4D_1\slash A_1)$ and $\e_{s\wp}\exp(-4D_2\slash A_3)$ from \eqref{*1}.
Clearly,
\begin{equation}\label{zak1}
\begin{split}
\e_{s\wp}\exp(-4D_1\slash A_1)
\leq&\e_0\exp(-4D_1\slash(M(s\wp) A_{1})),
\end{split}
\end{equation}
where
$$
M(s\wp)=\max_{\vartheta\in\Theta}e^{2\vartheta(s\wp)}=e^{2s\min_{\vartheta\in\Theta}\vartheta(\wp)}=e^{2s\gamma_\Theta(\wp)}.
$$
Proceeding exactly in the same way as in the proof of \cite[Lemma~6.2]{PUcm} we show that \eqref{zak1} is bounded by
\begin{equation}\label{zak11}
CM(s\wp)^{\o\gamma_\Theta(\alpha)}=Ce^{2s\gamma_\Theta(\wp)
\o\gamma_\Theta(\alpha)}.
\end{equation}

The expectation $\e_{s\wp}\exp(-4D_2\slash A_3)$ is similar.
Again, in the same way as in the proof of \cite[Lemma~6.2]{PUcm} we show that $\e_{s\wp}\exp(-4D_2\slash A_3)$ is bounded by
\begin{equation*}
CM_1(s\wp)^{\o\gamma_\Lambda(\alpha)},
\end{equation*}
where
\begin{equation*}
M_1(s\wp)=\max_{\lambda\in\Lambda}e^{2\lambda(s\wp)}=e^{2s\gamma_\Lambda(\wp)}.
\end{equation*}
Hence,
\begin{equation*}\label{a3zfalka}
\e_{s\wp}\exp(-4D_2\slash A_3)\leq Ce^{2s\gamma_\Lambda(\wp)\o\gamma_\Lambda(\alpha)}.
\end{equation*}
Now we estimate the second expectation on the right in \eqref{nu*0}.
\begin{equation*}
\begin{split}
\e_{s\wp}& A_0^{-1}A_1^{1\slash 2}\exp\left(-\frac{D}{ A_1}-\frac{D}{ A_3}\right)\\
\leq&\sum_{j=1}^n\e_{s\wp} A_0^{-1}A_{V,j}^{1\slash 2}\exp\left(-\frac{D}{ A_1}-\frac{D}{ A_3}\right)\\
=&\sum_{j=1}^n\e_{s\wp} A_{M,\Pi}^{-1\slash 2}\prod_{k\ne j}A_{V,k}^{-1\slash 2}\exp\left(-\frac{D}{ A_1}-\frac{D}{ A_3}\right)\\
=&\sum_{j=1}^ne^{-\sum_{\xi\in\Xi}\xi(s\wp)}e^{-\sum_{\vartheta\ne\vartheta_j}\vartheta(s\wp)}\e_{0} A_{M,\Pi}^{-1\slash 2}\prod_{k\ne j}A_{V,k}^{-1\slash 2}\exp\left(-\frac{D}{ A_1}-\frac{D}{ A_3}\right).
\end{split}
\end{equation*}
Since $s<0,$
\begin{equation*}
e^{-\sum_{\xi\in\Xi}\xi(s\wp)}e^{-\sum_{\vartheta\ne\vartheta_j}\vartheta(s\wp)}\leq e^{-\rho_0(s\wp)}.
\end{equation*}
To estimate
\begin{equation*}
\e_{0} A_{M,\Pi}^{-1\slash 2}\prod_{k\ne j}A_{V,k}^{-1\slash 2}\exp\left(-\frac{D}{ A_1}-\frac{D}{ A_3}\right)
\end{equation*}
we proceed as in \eqref{*1} and \eqref{eqperp1} and get the same estimate.
Hence, the estimate \eqref{estth61} holds on $K_1.$

Now we have to consider the set
\begin{equation*}
K_2=K\cap\{(m,v):\|m\|<\varepsilon,\;\|v\|\geq\varepsilon\}.
\end{equation*}
On this set the estimate from Theorem~\ref{ubpsigma} simplifies and we get that
\begin{equation}\label{poison2}
\nu^{s\wp}(x)\leq C\e_{s\wp} A_0^{-1}\exp\left(-\frac{D_1}{ A_1}\right)+C\e_{s\wp} A_0^{-1}A_1^{1\slash 2}\exp\left(-\frac{D}{ A_1}\right).
\end{equation}
Using Lemma~\ref{lperp}, \eqref{eqperp1}, \eqref{zak1} and \eqref{zak11} as above, we get the estimate
\begin{equation*}
\begin{split}
\e_{s\wp} A_0^{-1}\exp\left(-\frac{D_1}{ A_1}\right)\leq&\left(\e_{s\wp}( A_0^{-1})^2\right)^{1\slash 2}\left(\e_{s\wp}\exp\left(-\frac{2D_1}{ A_1}\right)\right)^{1\slash 2}\\
\leq&e^{-\rho_0(s\wp)}e^{s\gamma_\Theta(\wp)\o\gamma_\Theta(\alpha)}.
\end{split}
\end{equation*}
As in the previous case the second expectation in \eqref{poison2} has the same estimate as the first one.
Hence, the estimate \eqref{estth611} holds on $K_2.$ Finally, we consider the set
\begin{equation*}
K_3=K\cap\{(m,v):\|m\|\geq\varepsilon,\;\|v\|<\varepsilon\}.
\end{equation*}
On $K_3,$
\begin{equation}\label{poison3}
\nu^{s\wp}(x)\leq C\e_{s\wp} A_0^{-1}\exp\left(-\frac{D_2}{ A_3}\right)+C\e_{s\wp} A_0^{-1}A_1^{1\slash 2}\exp\left(-\frac{D_2}{ A_3}\right).
\end{equation}
We have
\begin{equation*}
\begin{split}
\e_{s\wp} A_0^{-1}\exp\left(-\frac{D_2}{ A_3}\right)
\leq&\left(\e_{s\wp}( A_0^{-1})^2\right)^{1\slash 2}\left(\e_{s\wp}\exp\left(-\frac{2D_2}{ A_3}\right)\right)^{1\slash 2}\\
\leq&e^{-\rho_0(s\wp)}e^{s\gamma_\Lambda(\wp)\o\gamma_\Lambda(\alpha)}.
\end{split}
\end{equation*}
Again, the second expectation in \eqref{poison3} has the same estimate as the first one. Thus \eqref{estth6111} is proved.
\end{proof}
\subsection{Proof of Theorem~\ref{newupper}}
Having Theorem~\ref{thnua}, we use the standard homogeneity argument as follows.
\begin{proof}[Proof of Theorem~\ref{newupper}]
It is clear that for $x\in N$ with the norm $|x|_\wp\leq 1$ we have $\nu(x)\leq C_{\wp}.$

Let $\delta_t^\rho=\Ad((\log t)\wp).$ Then $|\delta_t^\wp x|_\wp=t|x|_\wp.$ Let $x=\delta^\wp_{\exp(-s)}x_o$ with $|x_o|_\wp=1$ and $s<0.$ Then $|x|_\wp=e^{-s}>1.$ Let $K=\{x_o:\;|x_o|_\wp=1\}.$ By definition \eqref{defnua} of $\nu^{s\wp},$ we get
\begin{equation*}
\nu(x)=\nu(\delta^\rho_{\exp(-s)}x_o)=\nu((s\wp)^{-1}x_o(s\wp))=e^{\rho_0(s\wp)}\nu^{s\wp}(x_o),
\end{equation*}
where $\rho_0=\sum_j\vartheta_{j}+\sum_i\xi_{i},$ and the result follows from Theorem~\ref{thnua}.
\end{proof}
\section{Example}\label{example}
Consider $N=\mathcal H_n,$ the $2n+1$-dimensional Heisenberg group, which we realize as $\R^n\times\R^n\times\R$ with the Lie group multiplication given by
$$
(x_1,y_1,z_1)(x_2,y_2,z_2)=(x_1+x_2,y_1+y_2,z_1+z_2+x_1\cdot y_2).
$$

The corresponding Lie algebra $\mathfrak h_n$ is then spanned by the left invariant vector fields
\begin{equation*}
X_j=\partial_{x_j},\quad Y_j=\partial_{y_j}+x_j\partial_z,\quad  Z=\partial_z
\end{equation*}
where $1\le j\le n$.  Let $A=\R^k$ and let $\xi_{1,j}$, $\xi_{2,j}$, $\xi_3\in \lp\R^k\rp^*$, $1\le j\le n$, be such that
$$
\xi_{1,j}+ \xi_{2,j}=\xi_3
$$
independently of $j$.  For $x\in\R^n$, $a\in \R^k$, and $i=1,2$, we set
$$
e^{\xi_i(a)}x=(e^{\xi_{i,1}(a)}x_1,e^{\xi_{i,2}(a)}x_2,\dots,e^{\xi_{i,n}(a)}x_n).
$$
We then define an $A$ action on $\mathcal H_n$ by automorphisms of  $\mathcal H_n$ by
\begin{equation*}
a(x,y,z)a^{-1}=(e^{\xi_1(a)}x,e^{\xi_2(a)}y,e^{\xi_3(a)}z),
\end{equation*}
We then let $S= \mathcal H_n\rtimes A$.

Let $\o X_j$, $\o Y_j$, and $\o Z$ be, respectively, $X_j$, $Y_j$, and $Z$ considered as left invariant vector fields on $S$.  Then
$$
\o X_j=e^{\xi_{1,j}(a)}X_j,\quad \o Y_j=e^{\xi_{2,j}(a)}Y_j,\quad \o Z=e^{\xi_3(a)}Z.
$$

 We set
\begin{equation}\label{HeisLa}
\ali
\mathcal L_\alpha&=\sum_{j=1}^n\left(\o X_j^2+\o Y_j^2\right)
+\o Z^2+\Delta_\alpha\\
&=\sum_{j=1}^n\left(e^{2\xi_{1,j}(a)}X_j^2+e^{2\xi_{2,j}(a)}Y_j^2\right)
+e^{2\xi_3(a)}Z^2+\Delta_\alpha,\\
\eal
\end{equation}
where $\Delta_\alpha$ is defined in \eqref{defofdelta}.
\begin{example}\label{example1}
Consider the operator $\mathcal L_\alpha$ defined in \eqref{HeisLa} on $\mathcal H_n\rtimes A$ with $A=\R^2$ and $\xi_{1,j}=(1,0),$ $\xi_{2,j}=(0,1).$ Theorem 1.2 of~\cite{PUcm} gives
\begin{equation*}
\nu(x,y,z)\leq C(1+|(x,y,z)|_\wp)^{-\frac{C_1\rho_0(\wp)\gamma(\alpha)}{4}},
\end{equation*}
where
\begin{equation*}
\gamma(\alpha)=2\min(\alpha_1,\alpha_2)
\end{equation*}
for some constant $C_1$ which depends on $\wp$ and can be computed. Take $\wp=(1,2).$ We have $\rho_0=\sum_j\xi_{1,j}+\sum_j\xi_{2,j}+\xi_3,$ where $\xi_{i,j}(a)=a_i,$ $i=1,2,$ $j=1,\ldots,n.$ To compute $C_1$ we proceed similarly as in \cite[Example~1]{PUcm} and get
\begin{equation*}
\nu(x,y,z)\leq C(1+|(x,y,z)|_{(1,2)})^{-\frac{\min(\alpha_1,\alpha_2)}{2}}.
\end{equation*}
whereas Theorem~\ref{newupper} gives, for $\|(x,z)\|>1$ and $\|y\|>1,$
\begin{equation*}
\nu(x,y,z)\leq C(1+|(x,y,z)|_{(1,2)})^{-\frac{\alpha_1}{2}-\frac{\min(\alpha_1,\alpha_2)}{2}}.
\end{equation*}
Similarly, Theorem 1.1 of~\cite{PUpota} gives, for every $q>1,$
\begin{equation*}
\nu(x,y,z)\leq C_q(1+|(x,y,z)|_{\alpha})^{-\frac{2}{q}(\min(\alpha_1,\alpha_2))^2},
\end{equation*}
whereas Theorem~\ref{newupper} gives, for $\|(x,z)\|>1$ and $\|y\|>1,$
\begin{equation*}
\nu(x,y,z)\leq C(1+|(x,y,z)|_{\alpha})^{-\frac{\alpha_1^2}{2}-\frac{\alpha_2^2}{2}}
\end{equation*}
which is again a better estimate if we take for example an operator with $\alpha_1\sqrt{\frac{4}{q}-1}<\alpha_2.$
\end{example}

\end{document}